\documentclass[preprint]{elsarticle}
\usepackage{lineno}

\journal{}

\usepackage{slashbox}
\usepackage{amsmath}
\usepackage{amssymb}
\usepackage{amsthm}
\usepackage{graphicx}
\usepackage{epic,eepic,epsfig}
\usepackage{color}
\usepackage{subfigure}  
\usepackage{placeins}   
\usepackage{multirow}
\usepackage{epstopdf}
\usepackage{varwidth}
\usepackage{float}
\usepackage[table]{xcolor}

\newtheorem{theorem}{Theorem}[section]
\newtheorem{example}{Example}[section]
\newtheorem{lemma}{Lemma}[section]
\newtheorem{remark}{Remark}[section]

\definecolor{tabclr}{cmyk}{0,0,1,0}

\begin{document}
	
	\title{A BDDC method with an adaptive coarse space  for three-dimensional advection-diffusion problems }
	
	\author[HYNU]{Jie Peng}
	\ead{pengjie24@hynu.edu.cn}
	
	\author[XTU1,XTU2]{Shi Shu}
	\ead{shushi@xtu.edu.cn}
	
	\author[XTU1,XTU2]{Junxian Wang\corref{cor}}
	\ead{wangjunxian@xtu.edu.cn}
	
	\author[SCNU]{Liuqiang Zhong}
	\ead{zhong@scnu.edu.cn}
	
	\cortext[cor]{Corresponding author}
	\address[HYNU]{College of Mathematics and Statistics, Hengyang Normal University, Hengyang 421010, China}
	\address[XTU1]{School of Mathematics and Computational Science, Xiangtan University, Xiangtan 411105, China}
	\address[XTU2]{Hunan Key Laboratory for Computation and Simulation in Science and Engineering,
		Xiangtan University, Xiangtan 411105, China}
	\address[SCNU]{School of Mathematical Sciences, South China Normal University, Guangzhou 510631, China}

	\begin{abstract}

		The solution of nonsymmetric but positive definite (NSPD) systems arising from advection-diffusion problems is an important research topic in science and engineering. Balancing domain decomposition by constraints with an adaptive coarse space(adaptive BDDC) constitute a significant class of nonoverlapping domain decomposition methods, commonly used for symmetric positive definite problems. In this paper, we propose an adaptive BDDC method that incorporates a class of edge generalized eigenvalue problems based on prior selected primal constraints to solve NSPD systems from advection-diffusion problems. Compared with the conventional adaptive BDDC method for such systems, the proposed approach further reduces the number of primal unknowns. Numerical experiments show that although the iteration count increases slightly, the overall computational time is significantly reduced.

	\end{abstract}

	\begin{keyword}
	advection-diffusion equation, nonsymmetric, BDDC, adaptive constraints, generalized eigenvalue problem
		\MSC[2010] 65N30 \sep 65F10 \sep 65N55
	\end{keyword}
	
	\maketitle

	\section{Introduction}\label{sec:1}
	
	
	
	The advection-diffusion equation holds substantial significance in scientific and engineering fields, as it describes the transport of substances and energy in fluids caused by convection and diffusion. It is also a crucial mathematical model for analyzing fluid transport problems. The linear systems arising from finite element discretization of advection-diffusion problems are usually nonsymmetric but positive definite (NSPD). A series of domain decomposition methods have been proposed and analyzed for solving these nonsymmetric systems \cite{CaiXC91:41,ToselliA01:5759,GerardoTallec04:745,TuLi08:25,PengShu21:184}, among which the balancing domain decomposition by constraints (BDDC) method \cite{Dohrmann03:246,MandelDohrmann03:639,MandelDohrmann05:167,LiWidlund06:250}  is a notable example.
	
	
	In the work of Tu and Li \cite{TuLi08:25}, a BDDC method was developed and analyzed for advection-diffusion problems with constant coefficients. Robin boundary conditions were employed to construct the local subdomain bilinear forms, and the coarse-level primal variable space included standard subdomain vertex constraints, edge/face average continuity, and flux average constraints. When the diameters of the subdomains were sufficiently small, a convergence rate estimate for the generalized minimal residual (GMRES) method was obtained. As is well known, the convergence of the BDDC method can be affected by coefficient jumps, coefficient ratios, and geometric details \cite{PechsteinDohrmann17:273}. This motivates the investigation of BDDC methods based on adaptive coarse spaces.

	
	In the work \cite{PengShu21:184}, we extended the adaptive BDDC method, which is usually applied to symmetric positive definite (SPD) problems, to the stabilized finite element discretization systems of the advection-diffusion equations. This preconditioner constructs the coarse space by formulating generalized eigenvalue problems (GEP) for faces and edges, respectively. It is worth noting that the construction procedures for face and edge GEP are similar, but the number of primal unknowns on the edges remains relatively large. Solving the coarse problem associated with the primal unknowns is performed serially, and its size directly affects the computational efficiency of the iterative method. Therefore, provided that the iteration count does not increase significantly, it is crucial to reduce the number of degrees of freedom in the coarse space as much as possible to improve the computational efficiency of the iterative solver.

	%
	
	In recent years, a new approach for a more effective edge GEP was proposed in \cite{PechsteinDohrmann17:273}. By applying this idea, Kim and Wang \cite{KimWang20:1928} proposed and analyzed an adaptive BDDC preconditioner based on prior selected primal constraints for SPD systems of elliptic problems. Compared with the edge GEP in \cite{CalvoWidlund16:524,KimChung17:191}, the number of primal unknowns on edges was significantly reduced. Although the size of the GEP on edges is larger than that of the original ones and both singular value decomposition and QR decomposition are required, these computational costs are acceptable in a parallel computing environment.
	
	
	In this paper, we extend these new edge GEP to the adaptive BDDC method for solving NSPD systems arising from advection-diffusion problems. The convergence theory of the GMRES iteration preconditioned with the adaptive BDDC method is established. Numerical results show that the adaptive BDDC algorithm based on the new GEP on edges is robust and can effectively reduce the size of the coarse space without significantly increasing the iteration count. Furthermore, additional numerical experiments demonstrate that the new method exhibits substantial advantages in computational time in cases involving irregular subdomain partitions and highly varying or random viscosity.

	
	The remainder of this paper is organized as follows. Section 2 briefly introduces the model problem and its finite element variational formulation. Section 3 presents the adaptive BDDC method for NSPD systems arising from advection-diffusion problems, with a focus on the selection of primal constraints. Section 4 reports related numerical experiments, and Section 5 provides concluding remarks.

	\section{Model problem}\label{sec:2}
	
	
	Let $\Omega \subset \mathbb{R}^3$ be a bounded Lipschitz polyhedral domain.
	We consider the advection–diffusion problem given by
	\begin{eqnarray}\label{model-AD-equation}
		\left\{
		\begin{array}{rcll}
			Lu := - \nabla \cdot (\nu \nabla u) + \boldsymbol{a} \cdot \nabla u + c u &=&  f   &\mbox{in}~ \Omega,\\
			u &=& 0 & \mbox{on}~\partial \Omega,
		\end{array}
		\right.
	\end{eqnarray}
	where $\partial \Omega$ denotes the boundary of $\Omega$;
	$\nu \in L^{\infty}(\Omega)$ is the positive viscosity coefficient;
	$\boldsymbol{a} \in (L^{\infty}(\Omega))^3$ is the velocity field with $\nabla \cdot \boldsymbol{a} \in L^{\infty}(\Omega)$;
	$c \in L^{\infty}(\Omega)$ is the reaction coefficient;
	and $f \in L^2(\Omega)$ is the source term.

	%
	
	Define a tetrahedral mesh partition $\mathcal{T}_h$ of the domain $\Omega$, where $h$ denotes the mesh size. For each element $K \in \mathcal{T}_h$, we introduce the Peclet number, which reflects the ratio of the advection rate to the diffusion rate:
	\begin{align*}
		P_{K} = \frac{h_K \|\boldsymbol{a}\|_{K;\infty}}{2\nu},
	\end{align*}
	where $\|\boldsymbol{a}\|_{K;\infty} = \sup_{x \in K} |\boldsymbol{a}(x)|$ and $h_K$ is the diameter of $K$.
	
	For each $x \in K$, define the positive function
	\begin{align*}
		C(x) = \begin{cases}
			\frac{\tau h_K}{2\|\boldsymbol{a}\|_{K;\infty}} & \text{if } P_{K} \ge 1, \\
			\frac{\tau h_K^2}{4\nu} & \text{if } P_{K} < 1,
		\end{cases}
	\end{align*}
	where $\tau$ is a given constant, and we set $\tau = 0.7$ in the numerical experiments. Define $\widetilde{c}(x) = c(x) - \frac{1}{2} \nabla \cdot \boldsymbol{a}(x)$ for $x \in \Omega$, and assume that $\widetilde{c}(x) \ge c_0 > 0$, where $c_0$ is a positive constant.

	%
	%
	%

	Define the linear conforming finite element space $V(\mathcal{T}_h) \subset H_0^1(\Omega)$ on $\mathcal{T}_h$.
	The stabilized finite element discrete variational problem corresponding to the model problem \eqref{model-AD-equation} is formulated as follows (see \cite{ThomasLeopoldo89:173}): find $u_h \in V(\mathcal{T}_h)$ such that
	\begin{align}\label{stab-var-pb}
		a(u_h,v_h) := b(u_h,v_h) + z(u_h,v_h) = g(v_h), \quad \forall v_h \in V(\mathcal{T}_h),
	\end{align}
	where
	\begin{align*}
		b(u_h,v_h) &:= \int_{\Omega} \left( \nu \nabla u_h \cdot \nabla v_h + C(x) L u_h \, L v_h + \widetilde{c} u_h v_h \right) dx, \quad \forall u_h, v_h \in V(\mathcal{T}_h), \\
		z(u_h,v_h) &:= \frac{1}{2} \int_{\Omega} \left( \boldsymbol{a} \cdot \nabla u_h \, v_h - \boldsymbol{a} \cdot \nabla v_h \, u_h \right) dx, \quad \forall u_h, v_h \in V(\mathcal{T}_h), \\
		g(v_h) &:= \int_{\Omega} f v_h \, dx + \int_{\Omega} C(x) f L v_h \, dx, \quad \forall v_h \in V(\mathcal{T}_h).
	\end{align*}
	
	It can be shown that $b(\cdot,\cdot)$ is symmetric positive definite (SPD) and $z(\cdot,\cdot)$ is skew-symmetric. Therefore, the resulting discrete system is nonsymmetric positive definite (NSPD).
	
	In the following section, we design an adaptive BDDC preconditioner for this NSPD system, with a focus on the definition of primal constraints.

	\section{BDDC preconditioner with adaptive primal constraints}\label{sec:3}

	%
	%
	We decompose the mesh $\mathcal{T}_h$ into $N$ non-overlapping subdomains, i.e.
	$\bar{\mathcal{T}}_h = \bigcup_{i=1}^N \bar{\Omega}_i,$
	where each non-ring subdomain $\Omega_i$ is an aggregation of several mesh elements. Any two distinct subdomains are either disjoint or intersect only at faces, edges, or vertices on their interfaces. For precise definitions of faces, edges, and vertices, we refer to \cite{PechsteinDohrmann17:273}.
	
	Based on this partition, we introduce the Schur complement system to be solved by the adaptive BDDC-preconditioned generalized minimal residual (PGMRES) method, along with the necessary function spaces and operators for constructing the preconditioning operator.
	
	Let $V_i$ denote the restriction of $V(\mathcal{T}_h)$ to $\Omega_i$. Since $V(\mathcal{T}_h)$ is spanned by linear finite element basis functions associated with nodes, $V_i$ can be represented as the linear combination of basis functions corresponding to nodes inside $\Omega_i$ and the truncated basis functions on $\Gamma_i:=\partial\Omega_i\backslash \partial\Omega$.
	
	Define the local bilinear functional
	$$a_i(u_h,v_h) := b_i(u_h,v_h) + z_i(u_h,v_h),$$
	where
	\begin{align*}
		b_i(u_h,v_h) &:= \int_{\Omega_i} (\nu \nabla u_h\cdot \nabla v_h + C(x)Lu_h Lv_h + \widetilde{c}u_hv_h) dx,~\forall u_h, v_h \in V_i,\\
		z_i(u_h,v_h) &:= \frac{1}{2} \int_{\Omega_i} (\boldsymbol{a}\cdot \nabla u_h v_h - \boldsymbol{a} \cdot \nabla v_h u_h) dx,~\forall u_h, v_h \in V_i.
	\end{align*}
	
	We decompose $V_i$ into two parts: the function space spanned by the basis functions inside the subdomain and the function space associated with $\Gamma_i$, i.e.,
	$$V_i = V_I^{(i)} \oplus W_i,$$
	where $V_I^{(i)}$ denotes the space spanned by basis functions of internal nodes in $\Omega_i$, and $W_i$ denotes the space of discrete harmonic extension functions on $\Omega_i$, defined by
	
	$$W_i = \{w_i \in V_i: a_i(w_i,v_i) = 0,\forall v_i \in V_I^{(i)}\}.$$

	Let
	$$W=\prod_{i=1}^{N} W_{i}.$$
	Note that functions in $W$ are not required to satisfy continuity across subdomain interfaces.
	
	Introduce the global discrete harmonic extension space $\widehat{W} \subset W$ consisting of functions that are continuous across interfaces
		$$\widehat{W} = \{\widehat{w} \in W: \widehat{w}~\text{is continuous on}~\Gamma:=\bigcup\limits_{i=1}^N\Gamma_i\},$$
	and the partially coupled function space
	$$\widetilde{W} = \{\widetilde{w} \in W: \widetilde{w}~\text{satisfies the primal constriants}\}.$$
	The stiffness matrix $A_i$ and load vector $f_i$ of subdomain $\Omega_i$ can be defined using $a_i(\cdot,\cdot)$ and $g(\cdot)$. Dividing the unknowns into internal and interface unknowns, $A_i$ and $f_i$ take the block form
	
	$$
	A^{(i)}=\left(\begin{array}{cc}
		A_{I I}^{(i)} & A_{I B}^{(i)} \\
		A_{B I}^{(i)} & A_{B B}^{(i)}
	\end{array}\right),~~
	f^{(i)}=\left(\begin{array}{c}
		f_{I}^{(i)} \\
		f_{B}^{(i)}
	\end{array}\right),
	$$
	%
	where $I$ corresponds to internal unknowns and $B$ corresponds to interface unknowns of $\Omega_i$.
	
	The Schur complement matrix and corresponding right-hand side vector for the interface unknowns are
	
	$$
	S^{(i)}=A_{BB}^{(i)} - A_{BI}^{(i)} (A_{II}^{(i)})^{-1} A_{IB}^{(i)},~~
	g^{(i)}=f_{B}^{(i)} - (A_{II}^{(i)})^{-1} f_{I}^{(i)}.
	$$
	
	Let $\widetilde{R}^{(i)}$ denote the restriction operator from $\widetilde{W}$ to $W_i$. Then, the Schur complement matrix and right-hand side vector on the partially coupled space $\widetilde{W}$ are
	
	$$\widetilde{S} = \sum\limits_{i=1}^N (\widetilde{R}^{(i)})^T S^{(i)} \widetilde{R}^{(i)},~~\widetilde{g} = \sum\limits_{i=1}^N (\widetilde{R}^{(i)})^T g^{(i)}.$$

	Let $\widetilde{R}$ be the injection operator from $\widehat{W}$ to $\widetilde{W}$. Then, the Schur complement system for the original problem reads
	
	\begin{align}\label{Schur-complex-sys}
		\widetilde{R}^T \widetilde{S} \widetilde{R} \widehat{u} = \widetilde{R}^T \widetilde{g}.
	\end{align}
	%
	%
	where $\widehat{u}$ represents the restriction of the finite element solution $u_h$ to the internal interfaces.
	
	Finally, let
	$$\widetilde{D} = \sum\limits_{i=1}^N \widetilde{R}_i^T D_i \widetilde{R}_i$$
	be the scaling matrix on $\widetilde{W}$, where $D_i$ is defined on $W_i$ and consists of blocks $D_F^{(i)}$, $D_E^{(i)}$, and $D_V^{(i)}$ associated with the faces, edges, and vertices of $\partial\Omega_i$, respectively. These blocks satisfy the partition of unity condition, i.e., for $X = F, E, V$,
	
	\begin{align}\label{partition-unity-cond}
		\sum\limits_{\nu \in n(X)} D_X^{(\nu)} = I,
	\end{align}
	where $I$ is the identity matrix and $n(X)$ is the set of subdomain indices sharing $X$.
	
	From the above, the BDDC preconditioner for solving the Schur complement system \eqref{Schur-complex-sys} can be written as
	
	$$M_{BDDC}^{-1} = \widetilde{R}^T \widetilde{D} \widetilde{S}^{-1} \widetilde{D}^T \widetilde{R}.$$
	
	It is clear from this definition that the performance of the preconditioner depends on the choice of primal constraints and the scaling matrices. In this paper, we assume that the scaling matrix is chosen as the deluxe scaling matrix \cite{DohrmannPechstein12Tech}, and we will describe a method for selecting the primal constraints below.
	
	Suppose that the partially coupled function space $\widetilde{W}$ satisfies the standard subdomain vertex continuity constraints. We next introduce the primal constraints associated with faces and edges.
	
	To this end, we first decompose each Schur complement matrix $S^{(i)}$ ($i = 1, \dots, N$) into its symmetric and skew-symmetric parts:
	
	%
	%
	%
	%
	$$
	S^{(i)} = B^{(i)} + Z^{(i)},
	$$
	where
	$$
	B^{(i)}  = \frac{1}{2}(S^{(i)} + (S^{(i)})^T),~Z^{(i)} = \frac{1}{2}(S^{(i)} - (S^{(i)})^T).
	$$
	
	Rewrite the matrix $B^{(i)}$ in the form of the following block matrix
	$$
	B^{(i)} = \left(\begin{array}{ll}
		B_{XX}^{(i)} & B_{XC}^{(i)} \\
		B_{CX}^{(i)} & B_{CC}^{(i)}
	\end{array}\right),
	$$
	where $X$ corresponds to the block associated with the internal degrees of freedom of a face $F$ or an edge $E$, and $C$ corresponds to the remaining degrees of freedom. For simplicity of notation, we denote $B_{XX}^{(i)}$ (with $X = F$ or $E$) simply as $B_X^{(i)}$.
	
	Let
	$$\widetilde{B} = \sum\limits_{i=1}^N (\widetilde{R}^{(i)})^T B^{(i)} \widetilde{R}^{(i)},$$
	and define the important operators involved in the theoretical analysis, namely the averaging operator and the jump operator, as
	
	$$
	E_D := \widetilde{R}\widetilde{R}^T \widetilde{D}: \widetilde{W} \rightarrow \widetilde{W},~~~~P_D := I - E_D: \widetilde{W} \rightarrow \widetilde{W}.
	$$
	It can be seen from the analysis in \cite{TuLi08:25} that the key to the convergence analysis of the BDDC algorithm is to prove
	
	$$
	\|E_{D} \widetilde{w}\|_{\widetilde{B}}^2 \le C\Theta \|\widetilde{w}\|_{\widetilde{B}}^2,~\forall \widetilde{w} \in \widetilde{W} \backslash \{0\},
	$$
	and it suffices to prove that
	\begin{equation}\label{conv-proof-key}
		\|P_{D} \widetilde{w}\|_{\widetilde{B}}^2 \leq C\Theta \|\widetilde{w}\|_{\widetilde{B}}^2, \quad \forall \widetilde{w} \in \widetilde{W} \backslash \{0\},
	\end{equation}
	where $C$ is a positive constant, and $\Theta \ge 1$ is a given threshold.
	
	Following \cite{KimChung17:191,KimWang20:1928}, the analysis shows that the key to proving inequality \eqref{conv-proof-key} is to establish the following two estimates on each face $F$ and each edge $E$, respectively:
	
	\begin{align}\label{conv-proof-key-F}
		\sum_{i \in n(F)}\left\langle B_{F}^{(i)} \sum_{k \in n(F)} D_{F}^{(k)}\left(w_{i, F}-w_{k, F}\right), \sum_{k \in n(F)} D_{F}^{(k)}\left(w_{i, F}-w_{k, F}\right)\right\rangle \leq C \sum_{i \in n(F)}\left\langle B^{(i)} w_{i}, w_{i}\right\rangle,
	\end{align}
	\begin{align}\label{estimate-key-E-1}
		\sum_{i \in n(E)}\left\langle B_{E}^{(i)} \sum_{k \in n(E)} D_{E}^{(k)}\left(w_{i, E}-w_{k, E}\right), \sum_{k \in n(E)} D_{E}^{(k)}\left(w_{i, E}-w_{k, E}\right)\right\rangle \le C
		\sum_{i \in n(E)} \left\langle B^{(i)} w_{i}, w_{i}\right\rangle,
	\end{align}
	where $B_X^{(\nu)}$ is the block corresponding to the internal degrees of freedom of $X$ (with $X = F$ or $E$) in $B^{(\nu)}$, $n(X)$ denotes the set of subdomain indices sharing $X$,
	$w_\nu$ denotes the restriction of $\widetilde{w}$ to the unknowns on $\Gamma_\nu$,
	and $w_{\nu,X}$ $(X = F, E)$ denotes the restriction of $w_\nu$ to the unknowns in $X$.
	
	In the following, we present the adaptive primal constraints on faces and edges that ensure the above two inequalities hold.
	%
	
	Firstly, the adaptive primal constraints on faces are introduced.

		For a given face $F$, assume that $n(F) = \{i,j\}$. Noting that the scaling matrices satisfy the partition of unity condition \eqref{partition-unity-cond}, the left-hand side of \eqref{conv-proof-key-F} can be equivalently rewritten as
		
		\begin{align}\nonumber
			&\sum_{i \in n(F)} (w_{i, F} - \sum_{k \in n(F)} D_{F}^{(k)} w_{k, F})^T B_{F}^{(i)} (w_{i, F} - \sum_{k \in n(F)} D_{F}^{(k)} w_{k, F}) \\\nonumber
			=& \sum_{i \in n(F)} (w_{i, F} - w_{j, F})^T (D_{F}^{(j)})^T B_{F}^{(i)} D_{F}^{(j)} (w_{i, F} - w_{j, F}) \\\nonumber
			=& (w_{i, F} - w_{j, F})^T (D_{F}^{(j)})^T B_{F}^{(i)} D_{F}^{(j)} (w_{i, F} - w_{j, F}) +
			(w_{i, F} - w_{j, F})^T (D_{F}^{(i)})^T B_{F}^{(j)} D_{F}^{(i)} (w_{i, F} - w_{j, F}) \\\label{conv-proof-key-F-leftside}
			=& (w_{i, F} - w_{j, F})^T B_{F} (w_{i, F} - w_{j, F}),
		\end{align}
		where
		$$
		B_{F}=\left(D_{F}^{(j)}\right)^{T} B_{F}^{(i)} D_{F}^{(j)}+\left(D_{F}^{(i)}\right)^{T} B_{F}^{(j)} D_{F}^{(i)}.
		$$
		
		Denote the Schur complement of $B^{(\nu)}$ ($\nu = i,j$) with respect to the unknowns interior to the face $F$ by
		
		$$\widetilde{B}_F^{(\nu)} = B_{FF}^{(\nu)} - B_{FC}^{(\nu)}(B_{CC}^{(\nu)})^{-1} B_{CF}^{(\nu)},$$
		and let
		$$
		\widetilde{B}_{F}=\widetilde{B}_{F}^{(i)}: \widetilde{B}_{F}^{(j)},
		$$
		%
		where the notation $A:B = A(A+B)^{\dagger}B$ denotes the parallel sum of $A$ and $B$ as defined in \cite{AndersonDuffin69:576}.
		
		Using the minimal energy property of the Schur complement matrices $\widetilde{B}_F^{(\nu)}$ and the properties of the parallel sum, the right-hand side of \eqref{conv-proof-key-F} satisfies
		
		\begin{align}\nonumber
			w_{i}^T B^{(i)} w_{i} + w_{j}^T B^{(j)} w_{j}
			\ge & w_{i,F}^T \widetilde{B}_F^{(i)} w_{i,F} + w_{j,F}^T \widetilde{B}_F^{(j)} w_{j,F}\\\label{conv-proof-key-F-rightside}
			\ge & w_{i,F}^T \widetilde{B}_F w_{i,F} + w_{j,F}^T \widetilde{B}_F w_{j,F}.
		\end{align}

		Therefore, combining \eqref{conv-proof-key-F-leftside} and \eqref{conv-proof-key-F-rightside}, it suffices to impose the following estimate under suitable primal constraints on $F$ to ensure \eqref{conv-proof-key-F} holds:
		
		\begin{align}\label{conv-proof-key-F-key}
			(w_{i, F} - w_{j, F})^T B_{F} (w_{i, F} - w_{j, F}) \le C(w_{i,F}^T \widetilde{B}_F w_{i,F} + w_{j,F}^T \widetilde{B}_F w_{j,F}).
		\end{align}
		Finally, introduce the generalized eigenvalue problem
		\begin{align}\label{equ-4-8}
			B_{F} v_{F}=\lambda \widetilde{B}_{F} v_{F}.
		\end{align}


	Suppose the number of degrees of freedom on the face $F$ is $n_F$. For $k = 1,2,\dots,n_F$, let $\lambda_{F,k}$ be the eigenvalues of the generalized eigenvalue problem \eqref{equ-4-8}, with corresponding eigenvectors $v_{F,k}$. For a given threshold $\Theta_F \ge 1$, assume that
	
	$$
	\lambda_{F,1} \ge \lambda_{F,2} \ge \cdots \ge \lambda_{F,n_{\Pi}^{F}} \ge \Theta_F > \lambda_{F,n_{\Pi}^{F} + 1} \ge \lambda_{F,n_{\Pi}^{F} + 2} \ge \cdots
	\lambda_{F,n_{F}}.
	$$
	
		Denote $N_{F} = \{1,2,\cdots,n_{\Pi}^F\}$  and $N_{F}^{\perp} = \{n_{\Pi}^F+1,n_{\Pi}^F+2,\cdots,n_F\}$.
		We assume that the eigenvectors additionally satisfy the following orthogonality conditions:
		$$\left(\widetilde{B}_{F} v_{F, l}\right)^{T} v_{F,m} = 0,~l \in N_{F}~\mbox{and}~m \in N_{F}^{\perp},$$
		and
		$$\left(B_{F} v_{F, l}\right)^{T} v_{F,m} = \delta_{lm},~l,m \in N_{F} \cup N_{F}^{\perp},$$
		where $\delta_{lm}$ is Kronecker delta function.

	Further, the adaptive primal constraints will then be enforced on the unknowns in $F$,
	\begin{align}\label{primal-constraints-F}
		\left(B_{F} v_{F, l}\right)^{T}\left(w_{i, F}-w_{j, F}\right)=0, l \in N_{F}.
	\end{align}
		After performing a change of basis and using the properties of the eigenvectors, the above constraints allow the decomposition
		\begin{align}\label{function-decompose}
			w_{\nu, F} = (w_{\nu, F})_{\Pi} + (w_{\nu, F})_{\Delta},~~~\nu = i,j,
		\end{align}
		where $(w_{\nu,F})_\Pi$ and $(w_{\nu,F})_\Delta$ are called the coarse (or adaptive primal) and dual parts of $w_{\nu,F}$, respectively, satisfying
		$$
		(w_{\nu, F})_{\Pi} = \sum\limits_{l \in N_{F}} \langle B_{F}w_{\nu,F}, v_{F,l}\rangle v_{F,l},~
		(w_{\nu, F})_{\Delta} = \sum\limits_{l \in N_{F}^{\perp}} \langle B_{F}w_{\nu,F}, v_{F,l}\rangle v_{F,l}.
		$$
		
		%
		Then, using the properties of the eigenvectors, the coarse components on $F$ are identical across subdomains, i.e.,
		\begin{align}\label{primal-constraints-F-equiv}
			(w_{i, F})_{\Pi} = (w_{j, F})_{\Pi}.
		\end{align}

	With the above choice of adaptive primal unknowns for a face $F$, the following lemma holds.
	\begin{lemma}\label{lemma-F}
		Let $w_i$, $i \in n(F)$, satisfy the adaptive primal constraints \eqref{primal-constraints-F} on the face $F$. Then the following estimate holds:
		
		$$
		\sum_{i \in n(F)}\left\langle B_{F}^{(i)} \sum_{j \in n(F)} D_{F}^{(j)}\left(w_{i, F}-w_{j, F}\right), \sum_{j \in n(F)} D_{F}^{(j)}\left(w_{i, F}-w_{j, F}\right)\right\rangle \leq 2 \Theta_F \sum_{i \in n(F)}\left\langle B^{(i)} w_{i}, w_{i}\right\rangle,
		$$
		where $w_{i, F}$ denote the restriction of $w_{i}$ to the unknowns in the face $F$, and $\Theta_F \ge 1$ is any given threshold.	
	\end{lemma}
		\begin{proof}
			It suffices to prove \eqref{conv-proof-key-F-key}.	
			By using the decomposition \eqref{function-decompose}, the equality of coarse components \eqref{primal-constraints-F-equiv}, and the generalized eigenvalue problem \eqref{equ-4-8}, we have
			
			\begin{align}\nonumber
				(w_{i, F} - w_{j, F})^T B_{F} (w_{i, F} - w_{j, F})
				= &\langle B_F ((w_{i, F})_{\Delta} - (w_{j, F})_{\Delta}),((w_{i, F})_{\Delta} - (w_{j, F})_{\Delta}) \rangle\\\nonumber
				\le & \Theta_{F} \langle \widetilde{B}_{F} ((w_{i, F})_{\Delta} - (w_{j, F})_{\Delta}),((w_{i, F})_{\Delta} - (w_{j, F})_{\Delta})\rangle  \\\nonumber
				\le & 2 \Theta_{F} (\langle \widetilde{B}_{F} (w_{i, F})_{\Delta},(w_{i, F})_{\Delta} \rangle + \langle \widetilde{B}_{F} (w_{j, F})_{\Delta},(w_{j, F})_{\Delta}\rangle \\\nonumber	
				\le & 2 \Theta_{F} (\langle \widetilde{B}_{F} ((w_{i, F})_{\Pi} + (w_{i, F})_{\Delta}),((w_{i, F})_{\Pi} + (w_{i, F})_{\Delta}) \rangle \\\nonumber	
				&~~~~+ \langle \widetilde{B}_{F} ((w_{j, F})_{\Pi} + (w_{j, F})_{\Delta}),((w_{j, F})_{\Pi} + (w_{j, F})_{\Delta}) \rangle) \\\nonumber	
				= & 2 \Theta_{F} [(w_{i, F})^T \widetilde{B}_{F} w_{i, F} + (w_{j, F})^T \widetilde{B}_{F} w_{j, F}].
			\end{align}
		\end{proof}

	Next, we present the adaptive primal constraints on edge $E$ such that the estimate \eqref{estimate-key-E-1} holds.

	For ease of exposition, we assume that the set of subdomain indices sharing edge $E$ is $n(E) = \{1,2,3\}$. Noting that $\sum_{k=1}^3 D_E^{(k)} = I$, the left-hand side of \eqref{estimate-key-E-1} can be equivalently written as
	
	\begin{align}\label{estimate-key-E-2}
		\sum_{i=1}^3 (w_{i, E} - \sum_{k=1}^3 D_{E}^{(k)} w_{k, E})^T B_{E}^{(i)} (w_{i, E} - \sum_{k=1}^3 D_{E}^{(k)} w_{k, E}).
	\end{align}
	
		Since the number of subdomains sharing an edge is more than two, in order to construct coarse degrees of freedom using a generalized eigenvalue problem analogous to that on faces, we first estimate \eqref{estimate-key-E-2} as follows:
		\begin{align}\nonumber \label{estimate-key-E-1-leftside}
			&	\sum_{i \in n(E)} \left\langle B_{E}^{(i)}(w_{i, E} - \sum_{k \in n(E)} D_{E}^{(k)} w_{k, E}), (w_{i, E} - \sum_{k \in n(E)} D_{E}^{(k)} w_{k, E}) \right\rangle\\\nonumber
			= & \sum_{i \in n(E)}\left\langle B_{E}^{(i)} \sum_{k \in n(E) \backslash \{i\}} D_{E}^{(k)}\left(w_{i, E}-w_{k, E}\right), \sum_{k \in n(E) \backslash \{i\} } D_{E}^{(k)}\left(w_{i, E}-w_{k, E}\right)\right\rangle   \\\nonumber
			\le & (|n(E)| - 1)  \sum_{i \in n(E)} \sum_{k \in n(E) \backslash \{i\}} \left\langle B_{E}^{(i)} D_{E}^{(k)}\left(w_{i, E}-w_{k, E}\right), D_{E}^{(k)}\left(w_{i, E}-w_{k, E}\right)\right\rangle   \\
			= & (|n(E)| - 1)  \sum_{i \in n(E)} \sum_{k \in n(E) \backslash \{i\}} \left\langle (D_{E}^{(k)})^T B_{E}^{(i)} D_{E}^{(k)}\left(w_{i, E}-w_{k, E}\right), \left(w_{i, E}-w_{k, E}\right)\right\rangle       	
		\end{align}
		
		Denote the Schur complement matrix of $B^{(\nu)} (\nu \in n(E))$  with respect to the unknowns interior to the face $E$ as $\widetilde{B}_E^{(\nu)}$, i.e.
		$$\widetilde{B}_E^{(\nu)} = B_{EE}^{(\nu)} - B_{EC}^{(\nu)}(B_{CC}^{(\nu)})^{-1} B_{CE}^{(\nu)},$$
		and let
		$$
		\widetilde{B}_{E}=\widetilde{B}_{E}^{(1)}: \widetilde{B}_{E}^{(2)} : \widetilde{B}_{E}^{(3)}.
		$$
		
		Using the minimal energy property which satisfied by the Schur complement matrices $\widetilde{B}_E^{(\nu)} (\nu = 1,2,3)$ and the property which satisfied by the parallel sum, the right-hand side of \eqref{estimate-key-E-1} satisfies
		\begin{align}\label{estimate-key-E-1-rightside}
			\sum\limits_{i \in n(E)}w_{i}^T B^{(i)} w_{i}
			\ge  \sum\limits_{i \in n(E)} w_{i,E}^T \widetilde{B}_E^{(i)} w_{i,E}
			\ge  \sum\limits_{i \in n(E)} w_{i,E}^T \widetilde{B}_E w_{i,E}.
		\end{align}
		
		Therefore, using \eqref{estimate-key-E-1-leftside} and \eqref{estimate-key-E-1-rightside}, we can see that it suffices to make the following estimate holds under certain primal constraints on $E$
		to ensure \eqref{estimate-key-E-1} holds, i.e.
		\begin{align}\label{estimate-key-E-1-key}
			\sum_{i \in n(E)} \sum_{k \in n(E) \backslash \{i\}} \left\langle (D_{E}^{(k)})^T B_{E}^{(i)} D_{E}^{(k)}\left(w_{i, E}-w_{k, E}\right), \left(w_{i, E}-w_{k, E}\right)\right\rangle < C \sum\limits_{i \in n(E)} w_{i,E}^T \widetilde{B}_E w_{i,E}.
		\end{align}
		
		Introduce a generalized eigenvalue problem(\cite{KimChung17:191})
		\begin{align}\label{equ-4-8-E}
			B_{E} v_{E} = \lambda \widetilde{B}_{E} v_{E},
		\end{align}
		where
		$$
		B_{E} = \sum_{i \in n(E)} \sum_{k \in n(E) \backslash \{i\}} D_{E}^{(k)})^T B_{E}^{(i)} D_{E}^{(k)}.
		$$
		
		Suppose the number of degrees of freedom on the edge $E$ is $n_E$. For $k = 1,2,\cdots, n_E$, let $\lambda_{E,k}$ be the eigenvalue of the generalized eigenvalue problem \eqref{equ-4-8-E}, and its corresponding eigenvector be $v_{E,k}$. For a given threshold $\Theta_E \ge 1$, assume that
		$$
		\lambda_{E,1} \ge \lambda_{E,2} \ge \cdots \ge \lambda_{E,n_{\Pi}^{E}} \ge \Theta_E > \lambda_{E,n_{\Pi}^{E} + 1} \ge \lambda_{E,n_{\Pi}^{E} + 2} \ge \cdots
		\lambda_{E,n_{E}}.
		$$
		
		Denote $N_{E} = \{1,2,\cdots,n_{\Pi}^E\}$  and $N_{E}^{\perp} = \{n_{\Pi}^E+1,n_{\Pi}^E+2,\cdots,n_E\}$.
		We assume that the eigenvectors additionally fulfill the following requirements
		$$\left(\widetilde{B}_{E} v_{E, l}\right)^{T} v_{E,m} = 0,~l \in N_{E}~\mbox{and}~m \in N_{E}^{\perp},$$
		and
		$$\left(B_{E} v_{E, l}\right)^{T} v_{E,m} = \delta_{lm},~l,m \in N_{E} \cup N_{E}^{\perp},$$
%
		where $\delta_{lm}$ is Kronecker delta function.

		Further, the adaptive primal constraints will then be enforced on the unknowns in $E$,
		\begin{align}\label{primal-constraints-E}
			\left(B_{E} v_{E, l}\right)^{T}\left(w_{i, E}-w_{j, E}\right)=0, l \in N_{E}.
		\end{align}
		
		After performing a change of basis and combining with the conditions satisfied by the eigenvectors, the previously established constraints can be expressed directly in terms of the unknowns, i.e.
		\begin{align}\label{function-decompose-E}
			w_{\nu, E} = (w_{\nu, E})_{\Pi} + (w_{\nu, E})_{\Delta},~\nu \in n(E),
		\end{align}
		where $(w_{\nu, E})_{\Pi}$ and $(w_{\nu, E})_{\Delta}$ called the coarse part (or adaptive primal part) and dual part of $w_{\nu, E}$ respectively, and which satisfy
		$$
		(w_{\nu, E})_{\Pi} = \sum\limits_{l \in N_{E}} \langle B_{E}w_{\nu,E}, v_{E,l}\rangle v_{E,l},~
		(w_{\nu, E})_{\Delta} = \sum\limits_{l \in N_{E}^{\perp}} \langle B_{E}w_{\nu,E}, v_{E,l}\rangle v_{E,l}.
		$$
		
		%
		Then, by using the conditions satisfied by the eigenvectors, we can see the coarse components  on $E$ are identical, i.e.,
		\begin{align}\label{primal-constraints-E-equiv}
			(w_{i, E})_{\Pi} = (w_{j, E})_{\Pi}, i,j\in n(E).
		\end{align}
		
		With the above choice of adaptive primal unknowns for an edge $E$, the following lemma can be obtained.
		\begin{lemma}\label{lemma-E-1}
			For $w_{i}, i \in n(E)$, satisfying the adaptive primal constraints \eqref{primal-constraints-E} on the edge $E$, the following estimate holds,
			$$
			\sum_{i \in n(E)}\left\langle B_{E}^{(i)} \sum_{j \in n(E)} D_{E}^{(j)}\left(w_{i, E}-w_{j, E}\right), \sum_{j \in n(E)} D_{E}^{(j)}\left(w_{i, E}-w_{j, E}\right)\right\rangle \leq 2 \Theta_E \sum_{i \in n(E)}\left\langle B^{(i)} w_{i}, w_{i}\right\rangle,
			$$
			where $w_{i, E}$ denote the restriction of $w_{i}$ to the unknowns in the edge $E$, and $\Theta_E \ge 1$ is any given threshold.
		\end{lemma}
		
		\begin{proof}
			It suffices to prove \eqref{estimate-key-E-1-key}.	
			
			By using \eqref{function-decompose}, \eqref{primal-constraints-E-equiv} and \eqref{equ-4-8-E}, we can see that
			\begin{align}\nonumber
				&\sum_{i \in n(E)} \sum_{k \in n(E) \backslash \{i\}} \left\langle (D_{E}^{(k)})^T B_{E}^{(i)} D_{E}^{(k)}\left(w_{i, E}-w_{k, E}\right), \left(w_{i, E}-w_{k, E}\right)\right\rangle \\\nonumber
				= &\sum_{i \in n(E)} \sum_{k \in n(E) \backslash \{i\}} \left\langle (D_{E}^{(k)})^T B_{E}^{(i)} D_{E}^{(k)}\left((w_{i, E})_{\Delta}-(w_{k, E})_{\Delta}\right), \left((w_{i, E})_{\Delta}-(w_{k, E})_{\Delta}\right)\right\rangle\\\nonumber
				\le &2 \sum_{i \in n(E)} \sum_{k \in n(E) \backslash \{i\}} \left\langle (D_{E}^{(k)})^T B_{E}^{(i)} D_{E}^{(k)}\left((w_{i, E})_{\Delta}, (w_{i, E})_{\Delta}\right)\right\rangle \\\nonumber
				&+ 2 \sum_{i \in n(E)} \sum_{k \in n(E) \backslash \{i\}} \left\langle (D_{E}^{(k)})^T B_{E}^{(i)} D_{E}^{(k)}\left((w_{k, E})_{\Delta}, (w_{k, E})_{\Delta}\right)\right\rangle
				\\\nonumber
				\le &2 \sum_{i \in n(E)} \left\langle B_E \left((w_{i, E})_{\Delta}, (w_{i, E})_{\Delta}\right)\right\rangle
				\\\nonumber
				\le &2 \Theta_E \sum_{i \in n(E)} \left\langle \widetilde{B}_{E} \left((w_{i, E})_{\Delta}, (w_{i, E})_{\Delta}\right)\right\rangle
				\\\nonumber
				\le & 2 \Theta_{E} \sum_{i \in n(E)}(\langle \widetilde{B}_{E} ((w_{i, E})_{\Pi} + (w_{i, E})_{\Delta}),((w_{i, E})_{\Pi} + (w_{i, E})_{\Delta}) \rangle) \\\nonumber	
				= & 2 \Theta_{E} \sum_{i \in n(E)} (w_{i, E})^T \widetilde{B}_{E} w_{i, E}.
			\end{align}
		\end{proof}
		
		It can be observed that, compared to the construction of coarse degrees of freedom on faces,
		the construction process on edges involves additional approximations (see \eqref{estimate-key-E-1-leftside} for more details).
		This explains why the number of coarse degrees of freedom (or primal unknowns) on edges
		remains suboptimal in numerical experiments.
		To overcome this issue, inspired by \cite{PechsteinDohrmann17:273} and \cite{KimWang20:1928}, we propose a novel approach for constructing coarse degrees of freedom on edges.

	In the following, we will transform both the left and right expressions of \eqref{estimate-key-E-1} by introducing a new set of degrees of freedom.

	Let
	\begin{align}\label{equ-5-19}
		\check{w}_{k,E}=w_{k,E}-w_{1,E},~~ k=2,3, \quad \widehat{w}_{E}=\sum_{k=1}^{3} D_{E}^{(k)} w_{k,E}
	\end{align}
	and then we have
	$$
	\begin{aligned}
		w_{1,E} &=-\sum_{k=2}^{3} D_{E}^{(k)} \check{w}_{k,E}+\widehat{w}_{E}, \\
		w_{2,E} &=\left(I-D_{E}^{(2)}\right) \check{w}_{2,E}-D_{E}^{(3)} \check{w}_{3,E}+\widehat{w}_{E}, \\
		w_{3,E} &=-D_{E}^{(2)} \check{w}_{2,E}+\left(I-D_{E}^{(3)}\right) \check{w}_{3,E}+\widehat{w}_{E} .
	\end{aligned}
	$$
	
	Using the above identity, we can introduce  change of unknowns as follows,
	\begin{align} \label{equ-trans}
		w_{k,E}=T_{E}^{(k)}\left(\begin{array}{c}
			\check{w}_{2,E} \\
			\check{w}_{3,E} \\
			\widehat{w}_{E}
		\end{array}\right), \quad k=1,2,3,
	\end{align}
	where
	$$
	\begin{array}{l}
		T_{E}^{(1)}=\left(\begin{array}{lll}
			-D_{E}^{(2)} & -D_{E}^{(3)} & I
		\end{array}\right), \\
		T_{E}^{(2)}=\left(\begin{array}{lll}
			I-D_{E}^{(2)} & -D_{E}^{(3)} & I
		\end{array}\right),
	\end{array}
	$$
	and
	$$
	T_{E}^{(3)}=\left(\begin{array}{lll}
		-D_{E}^{(2)} & I-D_{E}^{(3)} & I
	\end{array}\right).
	$$

	Then, we can see
	\begin{align*}
		\left(\begin{array}{c}
			w_{1,E}-\sum\limits_{k=1}^{3} D_{E}^{(k)} w_{k,E}  \\
			w_{2,E}-\sum\limits_{k=1}^{3} D_{E}^{(k)} w_{k,E} \\
			w_{3,E}-\sum\limits_{k=1}^{3} D_{E}^{(k)} w_{k,E}	
		\end{array}\right)
		=
		\left(\begin{array}{c}
			w_{1,E}- \widehat{w}_E  \\
			w_{2,E}- \widehat{w}_E \\
			w_{3,E}- \widehat{w}_E	
		\end{array}\right)
		=
		\left(\begin{array}{cc}
			-D_{E}^{(2)} & -D_{E}^{(3)} \\
			I-D_{E}^{(2)} & -D_{E}^{(3)} \\
			-D_{E}^{(2)} & I-D_{E}^{(3)}
		\end{array}\right)
		\left(\begin{array}{c}
			\check{w}_{2,E}\\
			\check{w}_{3,E}	
		\end{array}\right).
	\end{align*}

	Using the new unknowns, we can rewrite \eqref{estimate-key-E-2} into
	\begin{align}\label{equ-5-20}
		\left(\begin{array}{l}
			\check{w}_{2 E} \\
			\check{w}_{3 E}
		\end{array}\right)^{T} M_{E}\left(\begin{array}{c}
			\check{w}_{2,E} \\
			\check{w}_{3,E}
		\end{array}\right),
	\end{align}
	where
	$$
	M_{E}=\left(\begin{array}{cc}
		-D_{E}^{(2)} & -D_{E}^{(3)} \\
		I-D_{E}^{(2)} & -D_{E}^{(3)} \\
		-D_{E}^{(2)} & I-D_{E}^{(3)}
	\end{array}\right)^{T}\left(\begin{array}{ccc}
		B_{E}^{(1)} & & \\
		& B_{E}^{(2)} & \\
		& & B_{E}^{(3)}
	\end{array}\right)\left(\begin{array}{cc}
		-D_{E}^{(2)} & -D_{E}^{(3)} \\
		I-D_{E}^{(2)} & -D_{E}^{(3)} \\
		-D_{E}^{(2)} & I-D_{E}^{(3)}
	\end{array}\right) .
	$$

	The right-hand side of \eqref{estimate-key-E-1} is greater than
	\begin{align}\label{estimate-key-E-3}
		\sum_{i=1}^3 \widetilde{w}_i^T \widetilde{B}^{(i)} \widetilde{w}_i,
	\end{align}
	where $\widetilde{w}_i := (w_{i,E}, w_{i,H})$
	represents the restriction of $w_i$ to the unknowns inside the edge $E$ and the prior selected primal unknowns (vertices and face primal unknowns). $\widetilde{B}_i$  is the Schur complement matrix obtained after eliminating the unknowns other than those inside $E$ and the prior selected primal unknowns from $B^{(i)}$, and
	$$
	\widetilde{B}^{(i)}=\left(\begin{array}{ll}
		\widetilde{B}_{E E}^{(i)} & \widetilde{B}_{E H}^{(i)} \\
		\widetilde{B}_{H E}^{(i)} & \widetilde{B}_{H H}^{(i)}
	\end{array}\right),
	$$
	here the notation $E$ denotes the blocks corresponding to the unknowns on $E$ and $H$ to the unknowns on the prior selected primal unknowns.
	
	From this, \eqref{estimate-key-E-3}  can be further equivalently represented as
	\begin{align}\nonumber
		&\left(
		\begin{array}{c}
			w_{1,E} \\
			w_{1,H}   \\
			w_{2,E}\\
			w_{2,H} \\
			w_{3,E}\\
			w_{3,H}
		\end{array}
		\right)^T
		\left(
		\begin{array}{cccccc}
			\widetilde{B}_{EE}^{(1)} & \widetilde{B}_{EH}^{(1)} & & & & \\
			\widetilde{B}_{HE}^{(1)} & \widetilde{B}_{HH}^{(1)} & & & & \\
			&&\widetilde{B}_{EE}^{(2)} & \widetilde{B}_{EH}^{(2)} & & \\
			&&\widetilde{B}_{HE}^{(2)} & \widetilde{B}_{HH}^{(2)} & & \\
			&&&&\widetilde{B}_{EE}^{(3)} & \widetilde{B}_{EH}^{(3)} \\
			&&&&\widetilde{B}_{HE}^{(3)} & \widetilde{B}_{HH}^{(3)} \\
		\end{array}
		\right)
		\left(
		\begin{array}{c}
			w_{1,E} \\
			w_{1,H}   \\
			w_{2,E}\\
			w_{2,H} \\
			w_{3,E}\\
			w_{3,H}
		\end{array}
		\right) \\\nonumber
		= &
		\left(
		\begin{array}{c}
			w_{1,E} \\
			w_{2,E}\\
			w_{3,E}\\\hline
			w_{1,H}   \\
			w_{2,H} \\
			w_{3,H}
		\end{array}
		\right)^T
		\left(
		\begin{array}{ccc|ccc}
			\widetilde{B}_{EE}^{(1)} &                     &                    &\widetilde{B}_{EH}^{(1)}
			&                     & \\
			& \widetilde{B}_{EE}^{(2)} &                    &                     &\widetilde{B}_{EH}^{(2)}  &\\
			&                     &\widetilde{B}_{EE}^{(3)} &                    &                     & \widetilde{B}_{EH}^{(3)} \\\hline
			\widetilde{B}_{HE}^{(1)} &                     &                    &\widetilde{B}_{HH}^{(1)}
			&                     & \\
			&\widetilde{B}_{HE}^{(2)}  &                    &
			& \widetilde{B}_{HH}^{(2)} & \\
			&                     &\widetilde{B}_{HE}^{(3)} &
			&                     & \widetilde{B}_{HH}^{(3)} \\
		\end{array}
		\right)
		\left(
		\begin{array}{c}
			w_{1,E} \\
			w_{2,E}\\
			w_{3,E}\\\hline
			w_{1,H}   \\
			w_{2,H} \\
			w_{3,H}
		\end{array}
		\right)\\\label{equ-5-18}
		=& \left(\begin{array}{l}
			w_{1,E} \\
			w_{2,E} \\
			w_{3,E} \\
			\widehat{w}_{H}
		\end{array}\right)^{T}
		\left(\begin{array}{llll}
			\widetilde{B}_{EE}^{(1)} & & &
			\dot{\widetilde{B}}_{EH}^{(1)} \\
			& \widetilde{B}_{EE}^{(2)} & & \dot{\widetilde{B}}_{EH}^{(2)} \\
			&& \widetilde{B}_{EE}^{(3)} &  \dot{\widetilde{B}}_{EH}^{(3)} \\
			\dot{\widetilde{B}}_{HE}^{(1)} & \dot{\widetilde{B}}_{HE}^{(2)} & \dot{\widetilde{B}}_{HE}^{(3)} & \widehat{B}_{H H}
		\end{array}\right)
		\left(\begin{array}{l}
			w_{1,E} \\
			w_{2,E} \\
			w_{3,E} \\
			\widehat{w}_{H}
		\end{array}\right)
	\end{align}
	where
	$$\dot{\widetilde{B}}_{EH}^{(1)} = (\widetilde{B}_{EH}^{(1)},{\bf 0},{\bf 0}),~
	\dot{\widetilde{B}}_{EH}^{(2)} = ({\bf 0},\widetilde{B}_{EH}^{(2)},{\bf 0}),~
	\dot{\widetilde{B}}_{EH}^{(3)} = ({\bf 0},{\bf 0},\widetilde{B}_{EH}^{(3)}),$$
	$$\dot{\widetilde{B}}_{HE}^{(1)} =
	\left(\begin{array}{c}
		\widetilde{B}_{HE}^{(1)} \\
		{\bf 0} \\
		{\bf 0}
	\end{array}
	\right), ~
	\dot{\widetilde{B}}_{HE}^{(2)} =
	\left(\begin{array}{c}
		{\bf 0}\\
		\widetilde{B}_{HE}^{(2)} \\
		{\bf 0}
	\end{array}
	\right),~
	\dot{\widetilde{B}}_{HE}^{(3)} =
	\left(\begin{array}{c}
		{\bf 0}\\
		{\bf 0}\\
		\widetilde{B}_{HE}^{(3)}
	\end{array}
	\right),~
	\widehat{w}_{H} =
	\left(\begin{array}{c}
		w_{1,H}   \\
		w_{2,H} \\
		w_{3,H}
	\end{array}
	\right),
	$$
	$$
	\widehat{B}_{H H} =	
	\left(
	\begin{array}{ccc}
		\widetilde{B}_{HH}^{(1)} 	&                     & \\
		& \widetilde{B}_{HH}^{(2)} & \\
		&                     & \widetilde{B}_{HH}^{(3)}
	\end{array}
	\right).
	$$
	
	From the transformation of basis  \eqref{equ-trans}, we can see that
	\begin{align*}
		\left(
		\begin{array}{c}
			w_{1,E} \\
			w_{2,E} \\
			w_{3,E} \\\hline
			\widehat{w}_H
		\end{array}
		\right)
		=
		\left(
		\begin{array}{c|c}
			T_{E}^{(1)} & \\
			T_{E}^{(2)} & \\
			T_{E}^{(3)} & \\\hline
			& I
		\end{array}
		\right)
		\left(
		\begin{array}{c}
			\check{w}_{2,E} \\
			\check{w}_{3,E} \\
			\widehat{w}_{E} \\\hline
			\widehat{w}_H
		\end{array}
		\right)
	\end{align*}
	and then we can rewrite the right-hand side of \eqref{equ-5-18} into
	\begin{align*}
		&	\left(
		\begin{array}{c}
			\check{w}_{2,E} \\
			\check{w}_{3,E} \\
			\widehat{w}_{E} \\\hline
			\widehat{w}_H
		\end{array}
		\right)^{T}
		\left(
		\begin{array}{c|c}
			T_{E}^{(1)} & \\
			T_{E}^{(2)} & \\
			T_{E}^{(3)} & \\\hline
			& I
		\end{array}
		\right)^T
		\left(\begin{array}{lll|l}
			\widetilde{B}_{EE}^{(1)} & & &\dot{\widetilde{B}}_{EH}^{(1)} \\
			& \widetilde{B}_{EE}^{(2)} & & \dot{\widetilde{B}}_{EH}^{(2)} \\
			&& \widetilde{B}_{EE}^{(3)} &  \dot{\widetilde{B}}_{EH}^{(3)} \\\hline
			\dot{\widetilde{B}}_{HE}^{(1)} & \dot{\widetilde{B}}_{HE}^{(2)} & \dot{\widetilde{B}}_{HE}^{(3)} & \widehat{B}_{H H}
		\end{array}\right)
		\left(
		\begin{array}{c|c}
			T_{E}^{(1)} & \\
			T_{E}^{(2)} & \\
			T_{E}^{(3)} & \\\hline
			& I
		\end{array}
		\right)
		\left(
		\begin{array}{c}
			\check{w}_{2,E} \\
			\check{w}_{3,E} \\
			\widehat{w}_{E} \\\hline
			\widehat{w}_H
		\end{array}
		\right) \\
		=&
		\left(
		\begin{array}{c}
			\check{w}_{2,E} \\
			\check{w}_{3,E} \\
			\widehat{w}_{E} \\\hline
			\widehat{w}_H
		\end{array}
		\right)^{T}
		\left(\begin{array}{c|c}
			\sum_{i=1}^3 (T_{E}^{(i)})^T \widetilde{B}_{EE}^{(i)} T_{E}^{(i)} & \sum_{i=1}^3 (T_{E}^{(i)})^T \dot{\widetilde{B}}_{EH}^{(i)} \\\hline
			\sum_{i=1}^3 \dot{\widetilde{B}}_{HE}^{(i)} T_{E}^{(i)} & \widehat{B}_{HH}
		\end{array}\right)
		\left(
		\begin{array}{c}
			\check{w}_{2,E} \\
			\check{w}_{3,E} \\
			\widehat{w}_{E} \\\hline
			\widehat{w}_H
		\end{array}
		\right)
	\end{align*}
	
	Let
	$$
	\widetilde{B}_{E}=		\left(\begin{array}{c|c}
		\sum_{i=1}^3 (T_{E}^{(i)})^T \widetilde{B}_{EE}^{(i)} T_{E}^{(i)} & \sum_{i=1}^3 (T_{E}^{(i)})^T \dot{\widetilde{B}}_{EH}^{(i)} \\\hline
		\sum_{i=1}^3 \dot{\widetilde{B}}_{HE}^{(i)} T_{E}^{(i)} & \widehat{B}_{HH}
	\end{array}\right).
	$$
	We can obtain a Schur complement of $\widetilde{B}_{E}$ by eliminating blocks corresponding to $\left(\widehat{w}_{E}, \widehat{w}_{H}\right)$ and denote it by $\widetilde{\widetilde{B}}_{E}$.
	
	From this, and it is easy to know that in order to make the estimation formula \eqref{estimate-key-E-1} hold, it suffices to make the following estimate holds under certain primal constraints on $E$
	\begin{align}\label{estimate-key-E-4}
		\left(\begin{array}{l}
			\check{w}_{2,E} \\
			\check{w}_{3,E}
		\end{array}\right)^{T} M_{E}\left(\begin{array}{c}
			\check{w}_{2,E} \\
			\check{w}_{3,E}
		\end{array}\right) \le C
		\left(\begin{array}{l}
			\check{w}_{2,E} \\
			\check{w}_{3,E}
		\end{array}\right)^{T} \widetilde{\widetilde{B}}_{E}\left(\begin{array}{c}
			\check{w}_{2,E} \\
			\check{w}_{3,E}
		\end{array}\right).
	\end{align}
	
	Introduce a generalized eigenvalue problem
	\begin{align}\label{equ-5-21}
		M_{E} \check{v}_E = \lambda \widetilde{\widetilde{B}}_{E} \check{v}_E.
	\end{align}
	Suppose the number of unknowns interior to $E$ is $n_E$. For $k = 1,2,\cdots, 2n_E$, let $\lambda_{E,k}$ be the eigenvalue of the generalized eigenvalue problem \eqref{equ-5-21}, and its corresponding eigenvector be $\check{v}_{E,k}$. For a given threshold $\Theta_E\ge 1$, assume that
	\begin{align}\label{equ-5-21-lambda}
		\lambda_{E,1} \ge \lambda_{E,2} \ge \cdots \ge \lambda_{E,n_{\Pi}^{E}} \ge \Theta_E {\color{blue}>} \lambda_{E,n_{\Pi}^{E} + 1} \ge \lambda_{E,n_{\Pi}^{E} + 2} \ge \cdots
		\lambda_{E,2n_{E}}.
	\end{align}
	
	Denote $N_E = \{1,2,\cdots,n_{\Pi}^E\}$,
	we then enforce the following constraints on the unknowns $w_{k,E}$,
	\begin{align}\label{equ-5-22}
		\left(M_{E} \check{v}_{E,l}\right)^{T}\left(\begin{array}{c}
			\check{w}_{2,E} \\
			\check{w}_{3,E}
		\end{array}\right)=0,~l\in N_E.
	\end{align}
	\begin{remark}
		Although the length of the eigenvectors obtained from the generalized eigenvalue problem \eqref{equ-5-21} is $2 n_E$,
		the basis transformation matrix associated with the edge $E$ of order $n_E$ can be constructed
		by using singular value decomposition (SVD) and QR decomposition methods,
		so that the primal constraint condition \eqref{equ-5-22} is satisfied.
		For specific details, we refer to \cite{KimWang20:1928}.
		
	\end{remark}
	
	With the above choice of adaptive primal unknowns for an edge $E$, we can obtain the desired bound:
	\begin{lemma}\label{lemma-E}
		For $w_i$, $i\in n(E)$, satisfying the adaptive primal constriants \eqref{equ-5-22} on the edge $E$, the following estimate holds,
		\begin{align*}
			\sum_{i \in n(E)}\left\langle B_{E}^{(i)} \sum_{k \in n(E)} D_{E}^{(k)}\left(w_{i, E}-w_{k, E}\right), \sum_{k \in n(E)} D_{E}^{(k)}\left(w_{i, E}-w_{k, E}\right)\right\rangle \le C  \Theta_E
			\sum_{i \in n(E)} \left\langle B^{(i)} w_{i}, w_{i}\right\rangle,
		\end{align*}	
		where $w_{i,E}$ denote the restriction of $w_i$ to the unknowns in the edge $E$ and $\Theta_E \ge 1$ is any given threshold.
	\end{lemma}
	\begin{proof}
		It suffices to prove \eqref{estimate-key-E-4}.
		
		Let $\check{w}_E$ denote $\left(\begin{array}{l}
			\check{w}_{2,E} \\
			\check{w}_{3,E}
		\end{array}\right)$, $V_{2E} = (\check{v}_{E,1},\check{v}_{E,2},\cdots,\check{v}_{E,n_{\Pi}^E})$, $V_{2E}^{\perp} = (\check{v}_{E,n_{\Pi}^E+1},\check{v}_{E,n_{\Pi}^E+2},\cdots,\check{v}_{E,2n_E})$.
		Then, there exists a $n_{\Pi}^E$-dimensional column vector $\check{w}_{E,\Pi}$ and a $2n_E - n_{\Pi}^E$-dimensional column vector $\check{w}_{E,\Delta}$ such that
		%
		$$\check{w}_{E} = (\check{w}_{E})_{\Delta} + (\check{w}_{E})_{\Pi},$$
		where 
		$$(\check{w}_{E})_{\Delta} = V_{2E}^{\perp} \check{w}_{E,\Delta},~(\check{w}_{E})_{\Pi} = V_{2E} \check{w}_{E,\Pi}.$$
		
		From this, by using the adaptive primal constraints \eqref{equ-5-22},
		\eqref{equ-5-21} and \eqref{equ-5-21-lambda}, we can see that
		\begin{align}\nonumber
			\left(\begin{array}{l}
				\check{w}_{2 E} \\
				\check{w}_{3 E}
			\end{array}\right)^{T} M_{E}\left(\begin{array}{c}
				\check{w}_{2,E} \\
				\check{w}_{3,E}
			\end{array}\right)
			= &\langle M_E ((\check{w}_{E})_{\Delta} + (\check{w}_{E})_{\Pi}),\check{w}_E \rangle\\\nonumber
			= &\langle M_E (\check{w}_{E})_{\Delta},\check{w}_E\rangle\\\nonumber
			\le & \Theta_{E} \langle \widetilde{\widetilde{B}}_{E} (\check{w}_{E})_{\Delta},\check{w}_E\rangle  \\\nonumber
			= & \Theta_{E} \langle \widetilde{\widetilde{B}}_{E} (\check{w}_{E})_{\Delta}, (\check{w}_{E})_{\Delta} + (\check{w}_{E})_{\Pi}\rangle\\\nonumber
			< & \Theta_{E} \langle \widetilde{\widetilde{B}}_{E} ((\check{w}_{E})_{\Delta}) + (\check{w}_{E})_{\Pi} , (\check{w}_{E})_{\Delta} + (\check{w}_{E})_{\Pi} \rangle \\\label{conclusion-1}
			=&C \Theta_{E} 	
			\left(\begin{array}{l}
				\check{w}_{2 E} \\
				\check{w}_{3 E}
			\end{array}\right)^{T} \widetilde{\widetilde{B}}_{E}\left(\begin{array}{c}
				\check{w}_{2,E} \\
				\check{w}_{3,E}
			\end{array}\right)
		\end{align}
		where \eqref{conclusion-1} can be obtained by utilizing the symmetric positive definiteness of $\widetilde{\widetilde{B}}_{E}$.
	\end{proof}

	By using Lemma \eqref{lemma-F} and Lemma \eqref{lemma-E},
	the following lemma can be obtained.
	\begin{lemma}\label{assum}
		Let $\Theta$ denote the maximum value between $\Theta_{E}$ and $\Theta_{F}$,
		we have
		\begin{align*}
			\langle \widetilde{B} (P_D \widetilde{w}), P_D \widetilde{w} \rangle \le C \Theta \langle \widetilde{B}\widetilde{w}, \widetilde{w}\rangle,~\forall \widetilde{w} \in \widetilde{W} \backslash \{0\},
		\end{align*}
		where the constant $C$ is only dependent on the number of faces and edges per subdomain and the number of subdomains sharing an edge.
	\end{lemma}

	Let $R^{(i)}$ be the restriction operator from $\widehat{W}$ to $W_i$, and introduce $$\widehat{B} = \sum\limits_{i=1}^N (R^{(i)})^T B^{(i)} R^{(i)}.$$	
	Denote $T := M_{BDDC}^{-1} \widetilde{R}^T \widetilde{S} \widetilde{R}$.
	Following the Theorem 7.15 in \cite{TuLi08:25}, and by using Lemma \ref{assum}, the convergence rate of the PGMRES algorithm can be derived and satisfies the following theorem.
	\begin{theorem}\label{theorem}
		There exists a positive constant $C$, which is independent of the subdomain diameter $H$ and the element size $h$, such that
		\begin{align*}
			c_1 \langle \widehat{B} \widehat{w}, \widehat{w} \rangle \le \langle \widehat{B} \widehat{w}, T \widehat{w}\rangle,~~
			\langle \widehat{B}(T \widehat{w}), T\widehat{w} \rangle \le C_2 \langle \widehat{B} \widehat{w}, \widehat{w}\rangle.
		\end{align*}
		Here
		$$
		c_1 = 1 - CH \frac{H}{h} \mu(H,h) \Theta,~~C_2 = C \Theta^2,
		$$
		where $\mu(H, h) = 1+log(H/h)$.
	\end{theorem}
	
	Define $\|v\|_{\widehat{B}}^2 = \langle \widehat{B}v, v\rangle$. From Theorem \ref{theorem} and the result in \cite{EisenstatElman83:345}, we know that the iterative convergence rate of the GMRES algorithm can be bounded by
	$$
	\frac{\|r_m\|_{\widehat{B}}}{\|r_0\|_{\widehat{B}}} \le \left(1 - \frac{c_1^2}{C_2}\right)^{m/2},
	$$
	where $r_m$ is the residual at step $m$ of the GMRES iteration applied to the operator $M_{BDDC}^{-1}$.

	\section{Numerical results}\label{sec:4}
	
	%
	In this section, we present numerical experiments to verify the effectiveness of the proposed adaptive BDDC method
	and compare it with the adaptive BDDC method proposed in \cite{PengShu21:184}.
	For ease of reference, we denote the adaptive BDDC method from \cite{PengShu21:184} as \textbf{ABDDC-OLD},
	and the method proposed in this paper as \textbf{ABDDC-NEW}.
	
	In the following tables, $(\text{pnumF}, \text{pnumE})$ denotes the number of primal unknowns on all faces and edges, respectively.
	GMRES iterations preconditioned with the adaptive BDDC methods are used to solve the discrete linear systems
	arising from the advection-diffusion problems.
	The iterations are terminated when the $L^2$-norm of the residual falls below $10^{-8}$ or the iteration count reaches $300$.
	In all experiments, we fix $\Theta_F = 1+\log(m)$ and $\Theta_E = 10$.
	These methods are implemented in MATLAB and run on a machine with an Intel(R) Xeon(R) Silver 4110 CPU @ 2.10GHz.
	
	\begin{example}\cite{GerardoTallec04:745}\label{example-2}
		Consider the model problem \eqref{model-AD-equation}, and we set $\Omega = (-0.5,0.5) \times (-0.5,0.5) \times (0,1)$, the reaction coefficient $c = 1$, the velocity field $\boldsymbol{a} = (-2\pi y, 2\pi x, \sin(2\pi x))$, the source term $f = 0$ and the boundary condition is given by
		\begin{align*}
			u = \left\{\begin{array}{ll}
				1, & z = 0;\\
				0, & \mbox{otherwise}.
			\end{array}
			\right.
		\end{align*}
	\end{example}
	
	We first consider numerical experiments on uniform meshes.
	The domain $\Omega$ is decomposed into $8$ uniform cubic subdomains,
	numbered in an anticlockwise helicoidal manner
	from $\Omega_1=(-0.5,0) \times (-0.5,0) \times (0,0.5)$ to $\Omega_8 = (-0.5,0) \times (0,0.5) \times (0.5,1)$.
	Let the viscosity coefficient $\nu$ be piecewise constant and $\nu_i = \nu|_{\Omega_i}$ for $i=1,\cdots,8$. We test and compare the two adaptive methods with the following tests:
	\begin{itemize}
		\item Test 1: $\nu_1 = \nu_4 = \nu_5 = \nu_8$, and $\nu_2 = \nu_3 = \nu_6 = \nu_7$, where $\bar{\Omega}_1 \cup \bar{\Omega}_4 \cup \bar{\Omega}_5 \cup \bar{\Omega}_8 = [-0.5,0] \times [-0.5,0.5] \times [0,1]$ and $\bar{\Omega}_2 \cup \bar{\Omega}_3 \cup \bar{\Omega}_6 \cup \bar{\Omega}_7 = [0,0.5] \times [-0.5,0.5] \times [0,1]$.
		
		\item Test 2: $\nu_1 = \nu_5 = \nu_6 = \nu_8$, and $\nu_2 = \nu_3 = \nu_4 = \nu_7$.
		
		\item Test 3: $\nu_1 = \nu_3 = \nu_6 = \nu_8$, and $\nu_2 = \nu_4 = \nu_5 = \nu_7$.
	\end{itemize}
	
	Each subdomain is partitioned into $m^3$ uniform cubes,
	and each cube is further divided into six uniform tetrahedra.
	
	Firstly, let $m = 6$. Table~\ref{ex-table-4} presents the iteration numbers and the number of primal unknowns on all faces and edges
	for both methods under various viscosities.
	From this table, we observe that both methods perform well,
	but the number of primal unknowns on all edges for the \textbf{ABDDC-NEW} method is slightly lower than that of \textbf{ABDDC-OLD} in Test 2 and Test 3.
	
	In Table~\ref{ex-table-5}, we set $\nu_1 = 1$ and $\nu_2 = 10^{-7}$.
	Both methods remain robust for different values of $m$,
	and the newly proposed method leads to a significant reduction in the number of primal unknowns on all edges.
	
	\begin{table}[H]
		\centering \caption{The results for various viscosities ($m=6$).}
		\label{ex-table-4}\vskip 0.1cm
		{\footnotesize
			\begin{tabular}{{|c|c|c|c|c|c|c|}}\hline
				&\multicolumn{3}{|c|}{ABDDC-OLD}           & \multicolumn{3}{|c|}{ABDDC-NEW}     \\\hline
				$\nu_1, \nu_2$     &Test 1   &Test 2       &Test 3     &Test 1     &Test 2       &Test 3    \\\hline
				$10^{-1}, 10^{-5}$ &10(0,3)     &9(0,6)       &9(0,9)
				&9(0,9)  &9(0,6)      &9(0,6)\\\hline
				$10^{-1}, 10^{-6}$ &10(0,3)     &9(0,6)       &9(0,9)
				&9(0,9)  &9(0,6)      &9(0,6)\\\hline
				$10^{-1}, 10^{-7}$ &10(0,3)     &9(0,6)       &9(0,9)
				&9(0,9)  &9(0,6)      &9(0,6)\\\hline
				$10^{3},  10^{-3}$ &9(6,21)    &9(5,22)      &3(4,30)
				&11(6,14)  &12(5,8)  &9(4,19)\\\hline
				$1,       10^{-7}$ &9(4,20)    &9(3,17)      &7(0,30)
				&10(4,9)  &11(3,6)       &10(0,6)\\\hline
			\end{tabular}
		}
	\end{table}

	\begin{table}[H]
		\centering \caption{The results for different number of subdomain problem size ($\nu_1 = 1, \nu_2 = 10^{-7}$)}
		\label{ex-table-5}\vskip 0.1cm
		{\footnotesize
			\begin{tabular}{{|c|c|c|c|c|c|c|}}\hline
				&\multicolumn{3}{|c|}{ABDDC-OLD}           & \multicolumn{3}{|c|}{ABDDC-NEW}     \\\hline
				$m$ &Test 1   &Test 2       &Test 3     &Test 1     &Test 2       &Test 3    \\\hline
				4  &8(4,12)      &9(3,10)       &7(0,18)
				&9(4,3)   &10(3,3)     &9(0,6)\\\hline
				8  &10(4,28)     &10(3,24)      &7(0,42)
				&11(4,10) &12(3,7)     &15(0,6)  \\\hline
				16 &12(6,61)     &11(5,55)      &8(4,90)
				&14(6,24)  &16(5,12)      &16(4,9)  \\\hline
			\end{tabular}
		}
	\end{table}
	
	Secondly, the numerical results of the two methods are tested for highly varying and random viscosity $\nu(x) \in (10^{-3}, 10^3)$ in Table  \ref{ex-table-random}. A similar performance can also be observed.
	\begin{table}[H]
		\centering \caption{The results for random viscosity ($\nu(x) = 10^r, r \in (-3, 3)$).}
		\label{ex-table-random}\vskip 0.1cm
		{\footnotesize
			\begin{tabular}{{|c|c|c|}}\hline
				$n(m)$     &ABDDC-OLD          &ABDDC-NEW     \\\hline
				2(4)       &10(4,8)     &11(4,0)\\\hline
				2(8)       &12(14,33)   &15(14,0)\\\hline
				2(12)      &13(15,51)   &16(15,9)\\\hline
				3(4)       &13(18,60)   &17(18,0) \\\hline
				3(8)       &17(64,201)  &25(64,18)\\\hline
				3(12)      &18(72,301)  &24(72,54)\\\hline
				4(4)       &15(52,187)  &24(52,18) \\\hline
				4(8)       &19(181,593) &36(181,112)\\\hline
				4(12)      &19(181,909) &30(181,309) \\\hline
			\end{tabular}
		}
	\end{table}
	
	Next, we present numerical results for irregular subdomain partitions.
	Using Metis, we first generate an initial tetrahedral mesh with a given mesh size $h$.
	Based on the principle of load balancing, we then aggregate the elements to obtain the irregular subdomain partitions.
	
	We also test the two methods for highly varying and random viscosity, $\nu(x) \in (10^{-3}, 10^3)$.
	Table~\ref{ex-table-cube} shows the iteration counts and the number of primal unknowns.
	It can be observed from the tables that although the iteration counts of the new method have increased,
	the number of primal unknowns on all edges for the new method is almost zero.
	This significantly reduces the size of the coarse space.
	
	Furthermore, Tables~\ref{ex-table-cube-time1} and \ref{ex-table-cube-time2} present the total times for solving these Schur complement systems using GMRES preconditioned by \textbf{ABDDC-OLD} and \textbf{ABDDC-NEW}, respectively.
	The times for the setup and solve phases of the generalized eigenvalue problem on edges are shown in parentheses.
	It can be seen that although the new method exhibits a slight increase in the number of iterations,
	the overall computational time is significantly reduced, demonstrating that the new method has a clear advantage in efficiency.

	\begin{table}[H]
		\centering \caption{Irregular subdomain partitions} 
	\label{ex-table-cube}\vskip 0.1cm
	{\footnotesize
		\begin{tabular}{{|c|c|c|c|c|c|c|}}\hline
			&\multicolumn{3}{|c|}{ABDDC-OLD}         & \multicolumn{3}{|c|}{ABDDC-NEW}     \\\hline
			\backslashbox[2cm]{$N$}{1/h} &20        &40           &80         &20         &40           &80    \\\hline
			8  &12(21,32)  &16(30,56)    &19(25,126) 	&17(21,0) &19(30,0)    &25(25,0)\\\hline
			27 &13(89,93)  &18(122,257)  &20(136,558)   &16(89,0) &21(122,0)   &26(136,0)\\\hline
			64 &12(161,110)&17(300,558)  &21(367,1177)  &15(161,2)&24(300,0)   &27(367,0) \\\hline
		\end{tabular}
	}
\end{table}

\begin{table}[H]
	\centering \caption{CPU times for ABDDC-OLD}
	\label{ex-table-cube-time1}\vskip 0.1cm
	{\footnotesize
		\begin{tabular}{{|c|c|c|c|}}\hline
			\backslashbox[2cm]{$N$}{1/h} &20        &40           &80             \\\hline
			8  &16.78(15.28,0.04)  &333.88(304.56,0.02)    &7816.22(6433.57,0.03) \\\hline
			27 &9.17(8.27,0.01)    &246.60(233.63,0.04)    &5179.07(4797.74,0.08)  \\\hline
			64 &3.82(2.94,0.01)   &166.72(157.59,0.05)    &4322.38(4142.04,0.14)   \\\hline
		\end{tabular}
	}
\end{table}

\begin{table}[H]
	\centering \caption{CPU times for ABDDC-NEW}
	\label{ex-table-cube-time2}\vskip 0.1cm
	{\footnotesize
		\begin{tabular}{{|c|c|c|c|c|c|c|}}\hline
			\backslashbox[2cm]{$N$}{1/h} &20        &40           &80         \\\hline
			8  &2.80(1.22,0.06)    &54.11(23.54,0.05)      &1751.30(531.28,0.09)\\\hline
			27 &1.84(0.64,0.12)    &30.67(16.31,0.25)      &710.67(315.80,0.43)\\\hline
			64 &1.58(0.36,0.22)    &23.44(11.30,0.82)      &450.33(250.97,1.28) \\\hline
		\end{tabular}
	}
\end{table}

\section{Conclusions}

In this paper, we propose and analyze an adaptive BDDC method for the NSPD system arising from advection-diffusion problems.
To address the issue that the generalized eigenvalue problem on edges was not satisfactory in previous studies,
we extend the generalized eigenvalue problem on edges based on the a priori selected primal constraints, originally designed for SPD problems,
to the NSPD systems of advection-diffusion problems.
Numerical experiments demonstrate that the new method achieves significant reductions in both the size of the coarse space and the computational time.

\section*{Acknowledgements}

	Jie Peng is supported by National Natural Science Foundation of China (Grant No. 12101250) and Science and Technology Projects in Guangzhou (Grant No. 202201010644).
	Shi Shu and Junxian Wang are supported by Science Challenge Project, China (Grant No. TZ2024009), National Natural Science Foundation of
	China (Grant No. 12371373), Key Project of Hunan Provincial Department of Education (Grant No. 22A0120), and Hunan Provincial
	Innovation Foundation for Postgraduate, China (Grant No. CX20220644).
	Liuqiang Zhong is supported by National Natural Science Foundation of China (Grant No. 12071160).

\bibliographystyle{abbrv}

\end{document}